\documentclass[12pt,leqno]{amsart}
\usepackage[dvips]{graphicx}
\usepackage{amsfonts}
\usepackage{amssymb}
\usepackage{amsmath}
\usepackage{amscd}
\usepackage{graphicx}
\usepackage[notcite,notref]{showkeys}
\def\DATE{\today}
\newtheorem{theorem}{Theorem}
\newtheorem{definition}[theorem]{Definition}
\newtheorem{corollary}[theorem]{Corollary}

\newtheorem{lemma}[theorem]{Lemma}

\newtheorem{proposition}[theorem]{Proposition}
\newtheorem{remark}[theorem]{Remark}

\newcommand\C{\mathbb{C}}

\newcommand\ch{\mathrm{char\:}}
\newcommand\chr{\mathrm{char\:}}

\newcommand\vp{\varphi}

\newcommand\og{\overline{\Gamma}}
\newcommand\ld{\ldots}
\newcommand\cd{\cdots}
\newcommand\op{\oplus}
\newcommand\la{\langle}
\newcommand\ra{\rangle}
\newcommand\hg{\widehat{G}}
\newcommand\Aut{\mathrm{Aut}\:}
\newcommand\Diagg{\mathrm{Diag}\:\Gamma}

\newcommand\ag{\mathrm{Aut}\;\mathfrak{g}}
\newcommand\dg{\mathrm{Der}\;\mathfrak{g}}
\newcommand\ad{\mathrm{ad}}
\newcommand\gr{\mathrm{gr}\;}

\newcommand\g{\mathfrak{g}}
\newcommand\G{\Gamma}

\newcommand\n{\mathfrak{n}}
\newcommand\K{\mathbb{K}}
\newcommand\Z{\mathbb{Z}}

\newcommand\pf{\noindent{\it Proof. }}
\newcommand\lr{\left\{ \begin{array}{l}}
\def\ds{\displaystyle}
\renewcommand{\mod}{\mathrm{mod}}

\title{Group gradings on filiform Lie algebras}
\author{Yuri Bahturin, Michel Goze, Elisabeth Remm}
\date{24 AV 5773}
\address{Department of Mathematics and Statistics
Memorial University of Newfoundland
St. John's, NL, A1C5S7. Universit\'{e} de Haute Alsace, LMIA, 4 rue des Fr\`{e}res Lumi\`{e}%
re, 68093 Mulhouse}
\email{bahturin@mun.ca, michel.goze@uha.fr, elisabeth.remm@uha.fr}

\begin{document}
\begin{abstract}We classify, up to isomorphism, gradings by abelian groups on nilpotent filiform Lie algebras of nonzero rank. In case of rank $0$, we describe conditions to obtain non trivial $\Z_k$-gradings.
\end{abstract}
\maketitle

\section{Introduction}\label{sII}

In this paper we are interested in describing (and where possible, classifying) abelian group gradings on nilpotent filiform Lie algebras of finite dimension over an algebraically closed field $\K$ of characteristic different from 2.

These algebras naturally arise in the study of affine structures of differential manifolds and their related Lie algebras. It was believed for some time that all nilpotent Lie algebras possess affine structures. However in 1992 Benoist \cite{Benoist} has produced an example of a 12-dimensional nonabelian nilpotent Lie algebra without affine structure. This algebra turned out to be filiform, that is of nil-index 11. On the other hand, it is well-known that the existence of an affine structure on a Lie algebra $\g$ over a field of characteristic 0 is guaranteed when $\g$ admits a $\Z$-grading with positive support. This makes it interesting to classify group gradings on filiform Lie algebras.

The class of filiform Lie algebras splits into two classes, as follows. Let us consider the automorphism group $A=\ag$ for $\g$. This is an algebraic group and we quickly show in this paper that $\ag$ is a solvable group, provided that $\dim\g\ge 4$. The maximal tori, that is the maximal subgroups of the form $(\K^\times)^\ell$, are conjugate in any algebraic group. The common rank $\ell$ of all maximal tori in $\ag$ is called the rank of the Lie algebra $\g$. Thus our two big classes are the Lie algebras of nonzero rank and Lie algebras of zero rank, also known as characteristically nilpotent.

Filiform Lie algebras of nonzero rank $\ell$, which is actually, either 1 or 2, have been described in \cite{G.K}. Our main result here is that each grading of any of the algebra $\g$ of nonzero rank is isomorphic to a factor-grading of one fixed grading by the group $\Z^\ell$. This grading is the eigenspace decomposition of $\g$ with respect to the action of a distinguished maximal torus.

In the case $\ell=2$, if $n=\dim\g$ is even, for each value $n$  there are just two pairwise non-isomorphic algebras $L_n$ and $Q_n$. If $n$ is odd, there is, up to isomorphism, only one algebra $L_n$. In the case $\ell=1$ there are two families of algebras usually denoted by $A_n^p$ and $B_n^p$ (see below in Section \ref{sCFLA}). For these algebras, we give a complete classification of all abelian group gradings, up to equivalence. Two gradings on $\g$ are equivalent if there is an automorphism of $\g$ which maps each component of the first grading to a component of the second. We also give the details of the gradings, such as the universal group and the saturated torus  of each grading. There are $\dfrac{(n-1)(n+2)}{2}$ pairwise inequivalent gradings in the case of $L_n$, $\dfrac{(n-1)(n+2)}{2}-1$ in the case of $Q_n$, $n+p-1$ in the case of $A_n^p$ and $n+p-2$ in the case of $B_n^p$.

The class of characteristically nilpotent algebras also splits into two subclasses: class of algebras which admit a nontrivial grading and class of algebras which does not admit any nontrivial gradings. Before this paper it was generally believed that characteristically nilpotent filiform Lie algebras do not admit any gradings because all derivations of such algebras are nilpotent. Actually, based on the description of characteristically nilpotent algebras in term of second cohomology group of $L_n$ or $Q_n$, for each $k\ge 2$, we can find an infinite family of characteristically nilpotent algebras which admit a nontrivial $\Z_k$-grading.

\section{Preliminaries}\label{sI}
In what follows $\K$ is an algebraically closed field of characteristic different from 2. For more details on group gradings of algebras see recent monograph \cite{EK}.

Let $\g$ be a Lie algebra over $\K$ and $S$ a set. An $S$-grading $\Gamma$ of $\g$ with support $S$ is a decomposition $\Gamma: \g= \bigoplus_{s \in S} \g_{s}$ of $\g$ as the sum on nonzero vector subspaces $\g_s$ satisfying the following condition. For any $s_1,s_2\in S$ such that $[\g_{s_1},\g_{s_2}]\neq 0$ there is $\mu(s_1,s_2)\in S$ such that  $[\g_{s_1},\g_{s_2}]\subset \g_{\mu(s_1,s_2)}$. If $\Gamma': \g= \bigoplus_{s' \in S'} \g_{s'}'$ is another grading of $\g$, we say that $\Gamma$ is \textit{equivalent} to $\Gamma'$ if there is an automorphism $\vp\in\Aut \g$ and a bijection $\sigma:S\to S'$ such that $\vp(\g_s)=\g_{\sigma(s)}'$. It follows that $\sigma(\mu(s_1,s_2))=\mu(\sigma(s_1),\sigma(s_2))$. If $S$ is a subset of a group $G$ such that $\mu(s_1,s_2)=s_1s_2$ in $G$, we say that $\Gamma$ is a \textit{group grading by the group $G$ with support $S$}. Such a group $G$ is not defined uniquely, but for any group grading there a universal grading group $U(\Gamma)$ such that any other grading group  $G$ of $\Gamma$ is a factor group of $U(\Gamma)$. The universal group is given in terms of generators and defining relations if one chooses $S$ as the set of generators and all the above equations $s_1s_2=\mu(s_1, s_2)$ as the set of defining relations.

Two group gradings $\Gamma$ and $\Gamma'$ of a Lie algebra $\g$ by groups $G$ and $G'$ with supports $S$ and $S'$ are called\textit{ weakly isomorphic} if they are equivalent, as above, and the map $\sigma:S\to S'$ extends to the isomorphism of groups $G$ to $G'$. The strongest relation is the \textit{isomorphism} of $G$-gradings. In this case both $\Gamma$ and $\Gamma'$ are gradings by the same group, they have the same support and the isomorphism of groups $\sigma$ is identity. As a result, we have $\vp(\g_\gamma)=\g'_\gamma$, for any $\gamma\in G$. In the case of group $G$-gradings with support $S$ we will write $\Gamma: \g= \bigoplus_{\gamma \in G} \g_{\gamma}$, assuming that $\g_\gamma=\{ 0\}$ if $\gamma\in G\setminus S$.

\medskip
Let   $\Gamma: \g= \bigoplus_{\gamma \in G} \g_{\gamma}$ and $\varepsilon:G\to K$ be a homomorphism of groups. We set $\g'_\tau=\bigoplus_{\varepsilon(\gamma)=\tau}\g_\gamma$, for any $\tau\in K$. Then we obtain the image-grading $\varepsilon(\Gamma)=\bigoplus_{\tau \in K} \g_{\tau}'$. If $\varepsilon$ is an onto homomorphism, we say that $\varepsilon(\Gamma)$ is a\textit{ factor-grading} of the grading $\Gamma$. Sometimes, people call $\varepsilon(\Gamma)$ a\textit{ coarsening} of $\Gamma$. A \textit{refinement} $\Gamma'$ of $\Gamma$ is an $H$-grading of $\g$ such that for any $\rho \in H$, there exists (a unique!) $\gamma \in G$ with $\g_{\rho}' \subset \g_{\gamma}$. In this case we obviously have the following:
\begin{equation}\label{ey1}
 \g_\gamma=\bigoplus_{\g_\rho'\subset\g_\gamma}\g_\rho'.
 \end{equation}
 A grading $\Gamma$ of $\g$ is called a \textit{fine grading} if doesn't admit proper refinements.

\begin{remark}If $\Gamma'$ is a refinement of $\Gamma$ then $\Gamma$ viewed as a $U(\Gamma)$-grading is a factor grading of  $\Gamma'$ viewed as a $U(\Gamma')$-grading.
\end{remark}
Indeed, if $s'$ is in the support $S'$ of $\Gamma'$ then $\g_{s'}'\subset\g_{s}$, for some uniquely defined $s$ in the support $S$ of $\Gamma$. Thus we have a well defined mapping $\psi:S'\to S:s'\to s$. Now we have a defining relation $s'_1s'_2=s'_3$ in the group $U(\Gamma')$ provided that $[\g_{s_1'}',\g_{s_2'}']=\g_{s_3'}'\neq\{ 0\}$. If $s_i=\psi(s_i')$, $i=1,2,3$ then $\{ 0\}\neq [\g_{s_1'}',\g_{s_2'}']\subset[\g_{s_1},\g_{s_2}]\subset\g_{s_3}$. Thus $\g_{s_3'}'$ has nontrivial intersection with $\g_{s_3}$ and so $\g_{s_3'}'\subset \g_{s_3}$, proving that $\psi(s_3')=s_3$. Now from what we have just written, it follows that $s_1s_2=s_3$ is a defining relation in $U(\Gamma)$. Thus, $\psi(s_1's_2')=\psi(s_3')=s_3=s_1s_2=\psi(s_1')\psi(s_2')$ and $\psi$ extends to a homomorphism of groups $U(\Gamma')$ onto $U(\Gamma)$. Finally, using equation (\ref{ey1}), we find that $\g_s=\bigoplus_{\psi(s')=s}\g_{s'}'$ because $\g_{s'}'\subset\g_s$ if and only if $\psi(s')=s$. We now have $\Gamma=\psi(\Gamma')$.

In the case of Lie algebras it is natural to assume that all groups involved in the group gradings are abelian. In fact, for many Lie algebras, like finite-dimensional semisimple ones, this is a satisfied (see \cite{BZ}) in the sense that the partial function $\mu:S\times S\to S$ appearing in the definition of the grading is symmetric or that the elements of the support in the group grading commute. So in what follows we always deal with gradings by abelian groups.  In addition, when we study finite-dimensional algebras, the supports of the gradings are finite sets, so our groups are finitely generated.

Now given a group $G$, we denote by $\hg$ the group of (1-dimensional) characters of $G$, that is the group of all homomorphisms $\chi:G\to F^\times$ where $F^\times$ is the multiplicative group of the field $F$.  If $\Gamma: \g=\bigoplus_{\gamma\in G}\g_\gamma$ is a grading of a Lie algebra $\g$ with a grading group $G$, there is an action of $\hg$ by semisimple automorphisms of $\g$ given on the homogeneous elements of $\g$ by $\chi\ast x=\chi(\gamma) x$ where $\chi\in\hg$ and $x\in\g_{\gamma}$. If $G$ is generated by the support $S$ of $\Gamma$, different characters act differently. Indeed, assume $\chi_1,\chi_2\in\hg$ are such that $\chi_1\ast x=\chi_2\ast x$, for any $x\in\g$. Choose any $s\in S$ and $0\neq x\in \g_s$. In this case $\chi_1(s)x=\chi_1\ast x=\chi_2\ast x=\chi_2(s)x$. As a result, $\chi_1(s)=\chi_2(s)$, for any $s\in S$. Since $\chi_1$ and $\chi_2$ are homomorphisms and $G$ is generated by $S$, we have $\chi_1=\chi_2$, as claimed. This allows us, in this important case, to view $\hg$ as a subgroup of $\ag$. We will view $\ag$ as an algebraic group. When we study finite-dimensional algebras, then $G$ is finitely generated abelian and so $\hg$ is the group of characters of a finitely generated abelian group. If $G\cong \Z^m\times A$, where $m$ is an integer, $m\ge 0$, and $A$ a finite abelian group, then $\hg\cong(F^\times)^m\times\widehat{A}$. Such subgroups of algebraic groups, consisting of semisimple elements, are called \textit{quasitori}.

A quasitorus is a generalization of the notion of the torus, which is an algebraic subgroup of $\ag$ isomorphic to  $\hg\cong(F^\times)^m$, for some $m$, called the dimension of the torus. A torus, which is not contained in a larger torus is called \textit{maximal}. The following result is classical.

\begin{theorem}\label{tmaxtor} In any algebraic group any two maximal tori are conjugate by an inner automorphism.
\end{theorem}
 Another well-known result, is often attributed to Platonov \cite{Platonof} (but see also \cite{SS})

\begin{theorem}\label{tquasi} Any quasi-torus is isomorphic to a subgroup in the normalizer of a maximal torus.
\end{theorem}


Thus, if we find a maximal torus $D$ in $\ag$ equal to its normalizer in $\ag$, then for any grading of $\g$ by a finitely generated abelian group $G$ there is $\vp\in \ag$ such that  $\vp\hg\vp^{-1}\subset D$.

Now every time we have a quasitorus $Q$ in $\ag$ there is root space decomposition of $\g$ with roots from the group of characters $\widehat{Q}$ for the group $Q$, the root subspace for $\lambda\in \widehat{Q}$ given by
$$\g_\lambda=\{x\in\g\;|\; \alpha(x)=\lambda(\alpha)x\mbox{ for any }\alpha\in Q\}.$$
If $Q\subset D$ then, by duality, $\widehat{Q}$ is a factor group of $\widehat{D}\cong \mathbb{Z}^m$ where $m=\dim D$. The root space decomposition by $D$ is the refinement of the root space decomposition by $Q$ and so the grading by $\widehat{Q}$ is a factor-grading of a grading by $\widehat{D}\cong\mathbb{Z}^m$.

Now assume that  we deal with a grading of $\g$ by a finitely generated abelian group $G$,  $\Gamma: \g=\bigoplus_{\gamma\in G}\g_\gamma$. Assume that $p=\ch F$ and write $G=G_p\times G_{p'}$, where $G_p$ is the Sylow $p$-subroup and $G_{p'}$ its complement in $G$ that has no elements of order $p$. If $\ch F=0$ then $G=G_{p'}$.  Let us consider the quasitorus $\hg\subset \ag$. Then there is $\vp\in\ag$ such that $\vp\hg\vp^{-1}\subset D$.  Let us switch to another $G$-grading $\vp(\Gamma):\g=\bigoplus_{\gamma\in G}\vp(\g_\gamma)$. The action of $\hg$ on $\g$ induced by $\vp(\Gamma)$ gives rise to another copy of $\hg$ in $\ag$, namely, $\vp\hg\vp^{-1}$ and now this subgroup is a subgroup in $D$. Replacing $D$ by another maximal torus $\vp^{-1}D\vp$ we may assume from the very beginning that $\hg\subset D$. Then, the root decomposition by  $D$ is a refinement of the root decomposition under the action of  $\hg$. Thus the grading $\Gamma':\g=\bigoplus_{\lambda\in\widehat{\hg}}\g_\lambda$ induced by this root decomposition resulting from  the action of $\hg$  is a factor-grading of the $\widehat{D}\cong\mathbb{Z}^m $ grading of $\g$ induced by by the action of the torus $D$.

Now the connection between the original grading $\Gamma$ and $\Gamma'$ is as follows. Let us consider the map $\psi: G\to \widehat{\hg}$ given by $\psi_g(\chi)=\chi(g)$. Clearly, this is a homomorphism of groups. The subgroup $G_p\subset\mathrm{Ker}\:\psi$ because $F$ has no nontrivial $p^{\mathrm{th}}$-roots of $1$. On the other hand, a standard manipulation with direct products and roots of 1 shows that for every nontrivial element $g\in G_{p'}$ there is a character of $ G_{p'}$, hence of $G$, which does not vanish on $g$. Thus $G_p=\mathrm{Ker}\:\psi$. Since every character is defined on the generators of the group $G_{p'}$ and the number of possible values of the images of generators in $F^\times$ does not exceed the order of generators,  $|\widehat{\hg}|\le |G_{p'}|$, the homomorphism $\psi$ is onto. Finally, since any $\lambda\in\widehat{\hg}$ is the image of a $\gamma\in G$ under the homomorphism $\psi$, we have
$$\g_\lambda=\g_{\psi(g)}=\{x\;|\;\chi(x)=\chi(g)x\}=\sum_{\psi(\gamma)=\lambda}\g_\gamma.$$
Thus the factor-grading of our original grading $\Gamma$ by the Sylow $p$-subgroup of $G$ turns out to be isomorphic to a factor-grading of a $\mathbb{Z}^m=\widehat{D}$-grading induced by the action of a maximal torus $D$ of $\ag$.

In the case where $G$ has no elements of order $p=\ch F$, we have that the original grading is isomorphic to a factor grading of a grading induced from the action of a maximal torus. Usually, this is some ``standard'' torus, and the grading induced by its action is also ``standard''.

We summarize the above discussion as follows.

\begin{theorem}\label{tMTp} Let $\Gamma: \g=\bigoplus_{g\in G}\g_g$ be a grading of a finite-dimensional algebra $\g$ over an algebraically closed field $\K$ by a finitely generated abelian group $G$. If $\chr \K=p>0$, let $G_p$ denote the Sylow $p$-subgroup of $G$. Consider the automorphism group $A=\ag$ of $\g$ and assume a maximal torus $D$ of $A$, of dimension $m$, equal to its normalizer in $A$. Then the factor-grading $\Gamma/G_p$ is isomorphic to a factor-grading of the standard $\Z^m$-grading of $\g$ induced by the action of $D$ on $\g$.
\end{theorem}

An important particular case is the following.

\begin{theorem}\label{tMTnp} Let $\Gamma: \g=\bigoplus_{g\in G}\g_g$ be a grading of a finite-dimensional algebra $\g$ over a algebraically closed field $\K$ by a finitely generated abelian group $G$. If $\chr \K=p>0$, assume $G$ has no elements of order $p$. Consider the automorphism group $A=\ag$ of $\g$ and assume a maximal torus $D$ of $A$, of dimension $m$, equal to its normalizer in $A$. Then $\Gamma$ is isomorphic to a factor-grading of the standard $\Z^m$-grading of $\g$ induced by the action of $D$ on $\g$.
\end{theorem}

Even more special is the following.
\begin{theorem}\label{tMTz} Let $\Gamma: \g=\bigoplus_{g\in G}\g_g$ be a grading of a finite-dimensional algebra $\g$ over a algebraically closed field of characteristic zero $\K$ by a finitely generated abelian group $G$. Consider the automorphism group $A=\ag$ of $\g$ and assume a maximal torus $D$ of $A$, of dimension $m$, equal to its normalizer in $A$. Then $\Gamma$ is isomorphic to a factor-grading of the standard $\Z^m$-grading of $\g$ induced by the action of $D$ on $\g$.
\end{theorem}

Since we completely classify gradings up to equivalence for certain classes of algebras, we quote some more results from \cite{EK}. Given a grading $\Gamma: \g= \bigoplus_{s \in S} \g_{s}$ of an algebra $\g$, the subgroup of the group $\ag$ permuting the components of $\Gamma$ is called the \textit{automorphism group} of the grading $\Gamma$ and denoted by $\Aut\Gamma$. Each $\vp\in\Aut\Gamma$ defines a bijection on the support $S$ of the grading: if $\vp(\g_s)=\g_{s'}$ then $s\mapsto s'$ is the desired permutation $\sigma(\vp)$, an element of the symmetric group $\mathrm{Sym}\; S$. The kernel of the homomorphism $\vp\mapsto\sigma(\vp)$ is called the \textit{stabilizer} of the grading $\Gamma$, denoted by $\mathrm{Stab}\;\Gamma$. Finally a subgroup of $\mathrm{Stab}\;\Gamma$, whose elements are scalar maps on each graded component of $\Gamma$ is called the \textit{diagonal group} of $\Gamma$ and denoted by $\Diagg$.

\medskip

\begin{definition}\label{dSQT} Let $Q\subset \Aut\Gamma$ be a quasitorus. Let $\Gamma$ be the eigenspace decomposition of $\g$ with respect to $Q$. Then the quasitorus $\Diagg$ in $\Aut\Gamma$ is called the saturation of $Q$. We always have $Q\subset\Diagg$. If $Q=\Diagg$ then we say that $Q$ is a \emph{saturated} quasitorus.
\end{definition}

A quasitorus $Q$ is saturated if and only is the group $\mathcal{X}(Q)$ of algebraic characters of $Q$ is $U(\Gamma)$, the universal group of $\Gamma$.

\begin{proposition}
The equivalence classes of gradings on $\g$ are in one-to-one correspondence with the conjugacy classes of saturated quasitori in $\Aut\g$.
\end{proposition}

Notice that if we already know that every quasitorus is conjugate to a subgroup of a fixed maximal torus and that two subgrous of the maximal torus are conjugate if and only if they are equal, we can say that the \textit{equivalence classes of gradings are in one-to-one correspondence with the saturated subquasitori of a fixed maximal torus.}


\section{Automorphisms of Filiform Lie Algebras}\label{sAFLA}

Our next goal will be to describe $\ag$ when $\g$ is a filiform nilpotent Lie algebra and the description of the quasitori in $\ag$. First we recall some definitions and results about filiform Lie algebras.

\subsection{Filiform Lie algebras}
 Let $\K$ be a field and $\g$ be a Lie algebra over $\K$. We denote by $\left\{\g^k\;|\; k=1,2,\ld\right\}$ the lower central series of $\g$ defined by $\g^1=\g$ and $\g^{k}=[\g^{k-1},\g]$, for $k=2,3,\ld$ One calls $\g$\textit{ nilpotent} if there is a natural $n$ such that $\g^{n+1}=\{ 0\}$. If $\g^n\ne\{ 0\}$ then $n$ is called the\textit{ nilpotent index} of $\g$. It is well-known that in this case a set of elements $\{ x_1,\ld,x_m\}\subset\g$  generates $\g$ if and only if the $\{x_1+\g^2,\ld,x_m+\g^2\}$ is the spanning set in the vector space $\g/\g^2$. As a result, a nilpotent Lie algebra with $\dim \g/\g^2\le 1$ is at most 1-dimensional. If $n>1$ then the nilpotent index of an $n$-dimensional nilpotent Lie algebra never exceeds $n-1$.

\begin{definition}\label{dFF} Given a natural number $n$, an $n$-dimensional Lie algebra $\g$ is called \emph{$n$-dimen- sional filiform} if the nilpotent index of $\g$ is maximal possible, that is, $n-1$. In this case we must have $\dim \g/\g^2=2$, $\dim \g^{k}/\g^{k-1}=1$ for $k=2,3,\ld, n-1$.
\end{definition}

In the situation described in Definition \ref{dFF},  the lower central series of $\g$ is ``thready'', whence the French name ``filiform''.
If we choose a basis $\{ e_1,e_2,\ld,e_{n-2},e_{n-1},e_n\}$ of $\g$ so that $\{ e_n\}$ is a basis of $\g^{n-1}$, $\{ e_{n-1}, e_n\}$ is a basis of $\g^{n-2}$, $\{ e_{n-2}, e_{n-1}, e_n\}$ is a basis of $\g^{n-3}$, etc., it is easy to observe that the center of $\g$ is 1-dimensional and equals $\g^{n-1}$.

\noindent Thus, the lower central sequence of $\g$ takes the form
$$\g= \g^1 \supset \g^2\supset \ld \supset \g^{n-1} =Z(\g)\supset\{ 0\},$$
where $Z(\g)$ is the center of $\g$ and all containments are proper. In any Lie algebra the lower central series is a \textit{central filtration} in the sense that $[\g^i,\g^j]\subset\g^{i+j}$.

\medskip

\begin{theorem}[\cite{Vergne}]\label{tV}
Any $n$-dimensional filiform $\K$-Lie algebra $\g$ admits an \emph{adapted basis} $\left\{ X_1, \ld , X_n  \right\}$, that is, a basis satisfying:
$$\left\{
\begin{array}{l}
[X_1,X_i]=X_{i+1}, \ i=2, \ld, n-1,\\
\left[X_2,X_3\right]=\sum_{k \geq 5} C_{23}^k X_k,\\
\left[X_i,X_{n-i+1}\right]=(-1)^{i+1}\alpha X_n , \mbox{ where } \alpha =0 \ \mbox{ when} \ n=2m+1,\\
\quad \g^i= \mathbb{K}\left\{ X_{i+1}, \ld , X_n \right\}\mbox{ for all } i\ge 2.
\end{array}
\right.$$
\end{theorem}

It follows that, in addition to the lower central filtration, that is, given by the lower central series, a filiform Lie algebra $\g$ has another central filtration, which is the refinement of the lower central one. This appears if one sets $V_{(i)}$, $i=1,2,\ld$ for the subspace of $\g$ spanned by
$\left\{ X_{i}, \ld , X_n \right\}$. For $i \geq 3$ we have $\g^i =V_{(i+1)}$ and hence
$$[V_{(i+1)},V_{(i+j+1)}] \subset V_{(i+j+1)} \ {\text{\rm if}} \ i+j < n, \ \ {\text{\rm and}}\ [V_{(i)},V_{(n-i+1)}]=\{ 0\}\mbox{ if }n\mbox{ is odd}.$$ This filtration will be useful in proving Theorem \ref{tAUT} below.

Now let us consider a collection of vector spaces $W_{i}=\g^{i}/\g^{i+1}$, $i=1,2,\ldots n-1$. The vector space direct sum $\gr\g=\displaystyle \bigoplus_{i=1}^{n-1}W_{i}$ becomes a Lie algebra if one defines the bracket of the elements by setting $[X+\g^{i+1},Y+\g^{j+1}]=[X,Y]+\g^{i+j+1}$, for $X\in\g^{i}$, $Y\in\g^{j}$, $1\le i,j\le n-1$. It follows from the above theorem that all the associated graded algebras for filiform Lie algebras are again filiform. They belong to one of the two types as follows.

\begin{enumerate}
\item[$L_n$:] Each of these algebras has an adapted basis $\{ X_1,X_2,\ldots,X_n\}$ such that $[X_1,X_i]=X_{i+1}$, for $i=2,3,\ldots,n-1$.
\item[$Q_n$:] Here $n=2m$. Each of these algebras has an adapted basis $\{ X_1,X_2,\ldots,X_n\}$ such that $[X_1,X_i]=X_{i+1}$, if $i=2,3,\ldots,n-1$, and $[X_j,X_{2m-j+1}]=(-1)^{j+1}X_{2m}$, if $2\le j\le m$.
\end{enumerate}

If $\gr\g\cong\g$ then we call $\g$ \textit{naturally graded} (by the group $\Z$). So there are only two types of naturally graded algebras: $L_n$ and $Q_n$.

\begin{corollary}
\cite{Vergne} Any naturally graded filiform Lie algebra is isomorphic to
\begin{itemize}
\item $L_n$ if $n$ is odd,
\item $L_n$ or $Q_n$ if $n$ is even.
\end{itemize}
\end{corollary}
We deduce that $L_n$ and $Q_n$ admits a $\Z$-grading with support $\{1,2,\ld,n\}$. In the next paragraph, we will define other non isomorphic $\Z$-gradings and we will determine the filiform Lie algebras admitting such gradings.

\subsection{The automorphisms group}

Let $\g$ be a filiform $n$-dimensional $\K$-Lie algebra. If $\{X_1,\ld,X_n\}$ is an adapted basis, then the set $\{X_1,X_2\}$ is a set of generators of $\g$. Moreover, the vector $X_1$ satisfies $\ad(X_1)^ {n-2}  \neq 0$ and for any $X \in \g$ we have $\ad(X)^ {n-1}=0.$
\begin{definition}\cite{goze-ancochea}
A vector $U\in\g-\g^2$ is called \textit{characteristic} if  $\ad(U)^{n-2}  \neq 0$.
\end{definition}
\begin{lemma}
Let $\sigma \in \ag$. If $X$ is a characteristic vector, then $\sigma(X)$ also is characteristic.
\end{lemma}

\pf Let $U$ be a characteristic vector. We can find an adapted basis $\{X_1,\ld,X_n\}$ of $\g$ such that $U=X_1$. We have
$$[\sigma(X_1),\sigma(X_i)]=\sigma(X_{i+1}), \ \ i=2,\ld,n-1.$$
Thus $\ad\sigma(X_1)^{n-2} \neq 0$ and $\sigma(X_1)=$ is a characteristic vector.

\begin{theorem}\label{tAUT}
Let $\g$ be a $n$-dimensional filiform $\K$-Lie algebra  with $n \geq 4$. Then the group $\ag$ is a solvable algebraic group, of toral rank at most 2.
\end{theorem}
\pf Let $\sigma \in \ag$ and $\{X_1,\ld,X_n\}$ an adapted basis of $\g$.  Then, from the previous lemma, $\sigma(X_1)=\sum\limits_{i \geq 1}a_{i1}X_i$ with $a_{11} \neq 0$. This implies that
$[\sigma(X_1),\sigma(X_i)]=\sigma(X_{i+1})$ belongs to the space generated by $\{X_{i+1},\ld,X_n\}$.
In particular,  $\sigma(X_3)=b_3X_3+\ld+b_nX_n$ with $b_3 \neq 0.$ Assume that  $\sigma(X_2)=\sum\limits_{i \geq 1}a_{i2}X_i$. Then
$[\sigma(X_2),\sigma(X_3)]= a_{12}b_3X_4+U$, where $U \in \la X_5,\ld,X_n\ra=V_{(5)}$. Since $\{X_1,\ld,X_n\}$ is  an adapted basis, we have $\left[X_2,X_3\right]=\sum_{k \geq 5} C_{23}^k X_k.$ It follows that $a_{12}b_3=0$; since $b_3 \neq 0$ we must have $a_{12}=0.$
  This shows that $\sigma(X_2) \in V_{(2)}$. Similarly we have $\sigma(X_i) \in V_{(i)}$.
 As a result, $\ag$ is a subgroup in the stabilizer of  the flag
 $$V_{(1)} \supset V_{(2)} \supset \ld \supset V_{(n-1)} \supset V_{(n)}\supset\{ 0\},$$
 which is a solvable group. The matrices of the transformations in $\ag$ with respect to the basis $\{X_1,\ld,X_n\}$ are lower triangular.\qed

\medskip
\noindent{\bf Remark} If $\dim\g=2$ or $3$, then $\ag$ contains a subgroup isomorphic to $GL(2, \K)$, hence not solvable.

\medskip

Assume that $\ag$ is solvable not nilpotent. Then the maximal abelian subalgebra $\frak{a}$ of $\dg$ is called a \textit{Malcev torus} and its dimension corresponds to the rank of $\g$.
\begin{corollary}\cite{G.K}
The rank of a filiform Lie algebra is equal to 0 or 1 or 2.
\end{corollary}
Thus, if $\g$ is filiform, the maximal torus $D$ of $\ag$ is isomorphic to $(\K^\times)^l$ with $0 \leq l \leq 2$.  

\section{On the classification of filiform Lie algebras}\label{sCFLA}

\subsection{Filiform algebras of rank $1$ or $2$}\label{ssFAR}

\begin{proposition}\label{gozekhakim}
Let $\frak{g}$ be a $n$-dimensional filiform $\K$-Lie algebra whose rank $r(\g)$ is not $ 0.$  Thus $r(\g)=2$ or $1$ and
\begin{enumerate}
\item If $r(\g)=2$, $\g$ is isomorphic to
\begin{enumerate}
\item $L_{n},\: n\geq 3$
$$
\lbrack X_{1},X_{i}\rbrack=X_{i+1}, \quad 1 \leq i \leq n-1
$$
\item $Q_{n},\: n=2m,\: m\geq 3$
$$
\begin{array}{ll}
\lbrack  Y_{1},Y_{i}\rbrack=Y_{i+1}, & 1 \leq i \leq n-2,\\
\lbrack Y_{i},Y_{n-i+1}\rbrack  =(-1)^{i+1}Y_{n}, & 2\leq 1 \leq m\\
\end{array}
$$
\end{enumerate}
\item If $r(\g)=1$, $\g$ is isomorphic to
\begin{enumerate}
\item $A_{n}^{p}(\alpha_{1},\dots,\alpha_{t}),\: n\geq 4,\: t=[\frac{n-p}{2}],\: 1 \leq p \leq n-4$
$$
\begin{array}{ll}
\lbrack X_{1},X_{i}\rbrack=X_{i+1}, \quad 1 \le i \leq n-1,\\
\lbrack X_{i},X_{i+1}\rbrack=\alpha_{i-1} X_{2i+p}, & 2 \leq i \leq t,\\
\lbrack X_{i},X_{j}\rbrack=a_{i-1,j-1} X_{i+j+p-1}, & 2 \leq i <j,\quad i+j \leq n-p+1\\
\end{array}
$$
\item $B_{n}^{p}(\alpha_1,\cdots,\alpha_{t}),\: n=2m,\: m\geq 3 ,\: \: 1 \leq p \leq n-5, \ t=[\frac{n-p-3}{2}]$
 $$
\begin{array}{ll}
\lbrack Y_{1},Y_{i}\rbrack=Y_{i+1}, & 1 \leq i \leq n-2,\\
\lbrack Y_{i},Y_{n-i+1}\rbrack  =(-1)^{i+1}Y_{n}, & 2\leq i\leq m\\
\lbrack Y_{i},Y_{i+1}\rbrack=\alpha_{i-1} Y_{2i+p}, & 2 \leq i \leq t+1\\
\lbrack Y_{i},Y_{j}\rbrack=a_{i-1,j-1} Y_{i+j+p-1}, & 2 \leq i <j,\: i+j \leq n-p\\
\end{array}
$$
\end{enumerate}
In both cases
$$
\left\{\begin{array}{l}
a_{i,i}=0,\\
a_{i,i+1}=\alpha_{i},\\
a_{i,j}=a_{i+1,j}+a_{i,j+1}.
\end{array}
\right.
$$
\end{enumerate}
where $\{X_1,\cdots,X_n\}$ is an adapted basis and $\{Y_1,\cdots,Y_n\}$ a quasi-adapted basis.
\end{proposition}

Note that if $\g$ is of the type $A_n^p$ then $\gr\g$ is of the type $L_n$. If $\g$ is of the type $B_n^p$ then $\gr\g$ is of the type $Q_n$.

\noindent{\bf{Remarks.}}
\begin{itemize}
  \item In this proposition, the basis used to define the brackets is not always an adapted basis. More precisely, it is adapted for  Lie algebras $L_n$ and $A_n^k$, and it is not adapted for Lie algebras $Q_n$ and $B_n^k$.
      But if $\{Y_1,\cdots,Y_n\}$ is a quasi-adapted basis, that is a basis which diagonalizes the semi-simple derivations, then the basis $\{X_1=Y_1-Y_2,X_2=Y_2,\cdots,X_n=Y_n\}$ is adapted.
  \item In \cite{G.K}, the proof of the main result is given for $\K=\C$. Without any further restrictions, we can extend this result to arbitrary fields $\K$ which are algebraically closed and of characteristic $0$. Since the proof is based on a simultaneous reduction of a semi-simple endomorphism and a nilpotent adjoint operator, this result can be extended to any field $\K$ which contains the eigenvalues of the semi-simple derivation. Thus the result is true  over any algebraically closed field.
  \item In the statement of the above result in \cite{G.K}, there is a third family denoted $C_n$. But all the Lie algebras of this family are isomorphic to $Q_n$. This error was noticed after the publication of the paper.
\end{itemize}

\subsection{Characteristically nilpotent Lie algebras}

\begin{definition}
A finite dimensional $\K$-Lie algebra $\g$ is called characteristically nilpotent if  any derivation of $\g$ is nilpotent. It is called characteristically unipotent if the group $Aut(\g)$ of the automorphisms of $\g$ is unipotent.
\end{definition}

A characteristically nilpotent Lie algebra has rank $0$ and the Lie algebra of derivations $\dg$ is nilpotent (but not necessarily characteristically nilpotent). In this case also $\ag$ is nilpotent. However, this group does not need  to be unipotent. It is unipotent if $\g$ is characteristically unipotent. Of course, if $\g$ is characteristically unipotent, it is characteristically nilpotent.

\noindent{\bf Examples}
\begin{enumerate}
\item
The simplest example \cite{campo}, denoted by $\n_{7,4}$ in terminology of \cite{GRbook}, is  7-dimensional and  given by
$$\left\{
\begin{array}{ll}
\left[  X_{1},X_{i}\right]    & =X_{i+1},\;2\leq i\leq6,\\
\left[  X_{2},X_{3}\right]    & =-X_{6},\\
\left[  X_{2},X_{4}\right]    & =-\left[  X_{5},X_{2}\right]  =-X_{7},\\
\left[  X_{3},X_{4}\right]    & =X_{7}.%
\end{array}
\right.$$
This Lie algebra is filiform. Any automorphism of $\n_{7,4}$ is unipotent. It is defined on the generators $X_1, X_2$ by
$$
\left\{
\begin{array}{l}
\sigma(X_1)=X_1+a_2X_2+a_3X_3+a_4X_4+a_5X_5+a_6X_6+a_7X_7,\\
\sigma(X_2)=X_2+b_3X_3+\frac{b_3^2-a_2}{2}X_4+b_5X_5+b_6X_6+b_7X_7.
\end{array}
\right.
$$
It follows that $\Aut(\n_{7,4})$ is a $10$-dimensional unipotent Lie group and $\n_{7,4}$ is also characteristically unipotent. Let us note that any $\sigma \in \Aut\n_{7,4}$ of finite order is equal to the identity.
\item The first example of characteristically nilpotent Lie algebra was given by  Dixmier and Lister \cite{DL}. It can be written as
$$\left\{
\begin{array}{llll}
\left[  X_{1},X_{2}\right]     =X_{5};& \left[  X_{1},X_{3}\right]
=X_{6};& \left[  X_{1},X_{4}\right]  =X_{7};& \left[  X_{1},X_{5}\right]
=-X_{8},\\
\left[  X_{2},X_{3}\right]     =X_{8};& \left[  X_{2},X_{4}\right]
=X_{6};& \left[  X_{2},X_{6}\right]  =-X_{7};& \left[  X_{3},X_{4}\right]
=-X_{5},\\
\left[  X_{3},X_{5}\right]     =-X_{7};& \left[  X_{4},X_{6}\right]  =-X_{8}.
\end{array}
\right.$$
It is a $8$-dimensional nilpotent  Lie algebra of nilindex $3$, and it is not filiform. Let us note that nilindex $3$
 is the
lowest possible nilindex for a characteristically nilpotent Lie algebra.
Its automorphisms group $\ag$ is not unipotent. For example, the linear map given by
$$
\left\{
\begin{array}{l}
\sigma(X_1)=X_5,\ \sigma(X_5)=X_1,\\
\sigma(X_2)=X_7,\ \sigma(X_7)=X_2,\\
\sigma(X_4)=X_8,\ \sigma(X_8)=X_4,\\
\sigma(X_3)=-X_3,\ \sigma(X_6)=-X_6
\end{array}
\right.
$$
is a non unipotent automorphism of $\g$ of order 2. In the case where $\ch\K\neq 2$, $\sigma$ defines a non-trivial $\Z_2$-grading $\G$ of $\g$:
$$
\G:\;\;\g=\la X_1+X_5, X_2+X_7,X_4+X_8\ra\op\la X_1-X_5, X_2-X_7,X_4-X_8, X_3, X_6\ra.
$$

\end{enumerate}

\section{Gradings of filiform Lie algebras}

\subsection{Standard gradings}
In each of the four types of filiform algebras introduced in the previous section there are standard gradings, as follows.

\begin{enumerate}
\item[\text(1)] If $\g=L_n$ then $\g=\bigoplus_{(a,b)\in\Z^2}\g_{(a,b)}$ where $\g_{(a,b)}=\{ 0\}$ except $\g_{(1,0)}=\la X_1\ra$, $\g_{(s-2,1)}=\la X_{s}\ra$, for all $s=2,\ldots, n$.
\item[\text(2)] If $\g=Q_n$ then $\g=\bigoplus_{(a,b)\in\Z^2}\g_{(a,b)}$ where $\g_{(a,b)}=\{ 0\}$ except $\g_{(1,0)}=\la(X_1+X_2)\ra $, $\g_{(s-2,1)}=\la X_{s}\ra$, for $s=2,\ldots ,n-1$, $\g_{(n-3,2)}=\la X_{n}\ra$.
\item[\text(3)] If $\g=A_n^p$ then $\g=\bigoplus_{a\in\Z}\g_a$ where $\g_a=\{ 0\}$ except $\g_1=\la  X_1\ra$, $\g_{s+p-1}=\la X_{s}\ra$, for $s=2,\ldots n$.
\item[\text(4)] If $\g=B_n^p$ then $\g=\bigoplus_{a\in\Z}\g_a$ where $\g_a=\{ 0\}$ except $\g_1=\la(X_1+X_2)\ra$, $\g_{s+p-1}=\la X_{s}\ra$, for $s=2,\ldots, n-1$, $\g_{n+2p-1}=\la X_{n}\ra$
\end{enumerate}
where $\{X_1,\cdots,X_n\}$ is an adapted basis.
Our main result says the following.

\begin{theorem}\label{tMT}
Let $\g$ be a finite-dimensional filiform $\K$-algebra of nonzero rank $r$. If $G$ is an abelian finitely generated group, then any $G$-grading of $\g$, whose support generates $G$, is isomorphic to a factor-grading of a standard grading by $\Z^r$.
\end{theorem}

\pf Each of the four cases above gives rise to a maximal torus $D$ in $\ag$. To describe $D$ we only need to indicate the action of an element of $D$ on the generators of $\g$, which will be $X_1$ and $X_2$ in the cases of $L_n$ and $A_n$ or $X_1+X_2$ and $X_2$ in the cases of $Q_n$ and $B_n$

\begin{enumerate}
\item[\text(1)] If $\g=L_n$ then $D=\{\vp_{u,t}\:\vert\: u,t\in\K^\times\}$ where $\vp_{u,t}(X_1)=uX_1$, $\vp_{u,t}(X_2)=tX_2$.
\item[\text(2)] If $\g=Q_n$ then $D=\{\vp_{u,t}\:\vert\: u,t\in\K^\times\}$ where $\vp_{u,t}(X_1+X_2)=u(X_1+X_2)$, $\vp_{u,t}(X_2)=tX_2$.
\item[\text(3)] If $\g=A_n^p$ then $D=\{\vp_u\:\vert\: u\in\K^\times\}$ where $\vp_u(X_1)=uX_1$, $\vp_u(X_2)=u^{p+1}X_2$.
\item[\text(4)] If $\g=B_n^p$ then $D=\{\vp_u\:\vert\: u\in\K^\times\}$ where $\vp_u(X_1+X_2)=u(X_1+X_2)$, $\vp_u(X_2)=u^{p+1}X_2$.
\end{enumerate}

\begin{lemma}
In each of the four cases in Theorem \ref{tMT}, the centralizer of $D$ is equal to $D$.
\end{lemma}

\begin{proof} If $\g$ is of the type $L_n$, then we have to determine all $\vp\in\ag$ such that $\vp_{u,t}\vp=\vp\vp_{u,t}$, for all $u,t\in \K^\times$. Notice that $\vp_{u,t}(X_1)=uX_1$ and $\vp_{u,t}(X_i)=u^{i-2}tX_i$, for all $i\ge 2$. Now let $\vp(X_i)=\sum_{j\ge i}a_{ji}X_j$, for $i=1,2,\ldots,n$. Then $\vp_{u,t}\vp(X_1)= a_{11}uX_1+\sum_{j\ge 2}a_{ij}u^{j-2}tX_j$ whereas $\vp\vp_{u,t}(X_1)=\sum_{j\ge 1}a_{ij}uX_j$. Also, $\vp_{u,t}\vp(X_i)=\sum_{j\ge i}a_{ij}u^{j-2}tX_j$ whereas $\vp\vp_{u,t}(X_i)=\sum_{j\ge i}a_{ij}u^{i-2}tX_j$, if $i\ge 2$. Thus we have $a_{j1}(u^{j-2}t-u)=0$, for all $j\ne 1$ and $a_{ji}(u^{j-2}t-u^{i-2}t)=0$, for all $j\ne i$. Here $u,t$ are arbitrary elements of an infinite set $\K^\times$. It follows that all $a_{ji}$ are zero as soon as $j\ne i$. Notice that in any case, it follows from $[X_1,X_i]=X_{i+1}$ for  $i\ge 2$, that $a_{ii}=a_{11}^{i-2}a_{22}$. As a result, $\vp=\vp_{a_{11},a_{22}}$, proving that the centralizer of $D$ is indeed, $D$ itself.

Now assume $\g$ is of the type $Q_n$. Then, instead of comparing the values of the sides of $\vp_{u,t}\vp=\vp\vp_{u,t}$ at $X_1,X_2,\ldots,X_n$ we can compare them on the quasi-adapted basis $X_1+X_2,X_2,\ldots,X_n$.  Then we obtain $a_{ji}=0$, for $j\ne i$, $i\ge 2$. At the same time,
$$
\vp(X_1+X_2)=a_{11}X_1+\sum_{j\ge 2}(a_{j1}+a_{j2})X_j=a_{11}(X_1+X_2)+(a_{21}+a_{22}+a_{11})X_2+\sum_{j\ge 3}(a_{j1}-a_{j2})X_j.
$$
Applying the same argument, as before, we obtain $a_{j1}=a_{j2}=0$, for $j\ge 3$ and
$a_{21}=-a_{11}-a_{22}$. Hence $\vp(X_1+X_2)=a_{11}(X_1+X_2)$. Thus, $\vp=\vp_{a_{11},a_{22}}$, as previously.

Next, assume $\g$ is of one of the types $A_n^p$, where $p\ge 1$. In this case we can repeat the argument of the case $L_n$, bearing in mind that $t=u^{p+1}$. Then we will have equations $a_{j1}(u^{j-2}u^{p+1}-u)=0$, or  $a_{j1}(u^{j+p-2}-1)=0$, for all $j\ne 1$ and $a_{ji}(u^{j-2}u^{p+1}-u^{i-2}u^{p+1})=0$, or $a_{ji}(u^{j}-u^{i})=0$ for all $j\ne i$ and all $u\in \K^\times$. Since $j+p-2\ne 0$, we again can make the same conclusion $a_{ji}=0$, for all $j\ne i$.

The case $B_n^p$, where $k\ge 1$, is reduced to the case $Q_n$ in the same manner as $A_n^p$ to $L_n$.

Thus the proof of our lemma is complete.
\end{proof}

To complete the proof of the Theorem \ref{tMT} we need only to refer to Theorem \ref{tMTp} from the Introduction.

\section{Classification of gradings on filiform algebras of nonzero rank, up to equivalence}\label{sCGR2}

Recall that two gradings of an algebra are equivalent if there is an automorphism of this algebra permuting the components of the grading. In what follows we use Theorem \ref{tMT} according to which any grading of a filiform Lie algebra of nonzero rank is isomorphic, hence equivalent, to a grading where the elements of an adapted (case $L_n$ and $A_n^p$) or quasi adapted (case $Q_n$ and $B_n^p$) basis are homogeneous. Let $G$ be the grading group and, for each $i=1,2,\ld,n$, $d_i$ denote the degree of the $i$th element of this basis. So if $\{ X_1,X_2,\ldots,X_n\}$ is the adapted basis of $\g$ then $d_i=\deg X_i$, $i=1,2,\ld,n$, in the case of $L_n$ and $A_n^p$ and $d_i=\deg Y_i$, $i=1,2,\ld,n$, in the case of $Q_n$ and $B_n^p$, where $Y_1=X_1+X_2$ and $Y_i=X_i$, for $i=2,\ld,n$. Since $\g$ is generated by $X_1,X_2$ (respectively, $Y_1,Y_2$) it follows that knowing $a=d_1$ and $b=d_2$ automatically gives values for the remaining $d_i$, $i=3,\ld,n$.

\subsection{Gradings on $L_n$ and $A_n^p$}\label{ssLA}

Let $\{X_1,X_2,\cdots,X_n\}$ be an adapted basis of a filiform Lie algebra $\g$ of type $L_n$ or $A_n^p$. Since $[X_1,X_i]=X_{i+1}$ for $i=2,\cdots,n-1$, we know that $d_i=a^{i-2}b$ for $i=3,\cdots,n$. However, if $\g$ is of the type $A_n^p$, we already have $b=a^{p+1}$. So in this case, $d_i=a^{p+i-1}$. In the case of $L_n$, the universal group of any grading is the factor group of the free abelian group with free basis $a,b$ by the relations satisfied by $a,b$, while in the case of $A_n^p$ this is a factor-group of the free abelian group of rank 1 or rank 2 but we have to consider, in the latter case, that $b=a^{p+1}$.

Let us first consider the case of $ L_ n $.  Any grading is a coarsening of the standard grading, which we denote by $\Gamma_{\mathrm{st}}$. If all $d_ i$ are pairwise different then there is no coarsening and we have the standard grading.

\textsc{Case 1.} If $d_1=d_l $, for $2\le l\le n $ then the grading is a coarsening of the grading
$$
\Gamma^l_0: \g=\langle X_2\rangle\oplus\cd\op\la X_1,X_l\ra\op\cd\op\la X_n\ra.
$$
Since this is indeed a grading of $\g$, our claim follows. The universal group of this grading is the factor group of the free abelian group generated by $ a, b $ by a single relation $ a= a ^ { l - 2 } b $, that is, the group $ \Z $.

\textsc{Case 2.} If $d_i=d_j$, for $2\le i<j\le n $ then the grading is a coarsening of a grading
$$
\Gamma_k^0:\g=\la X_1\ra\oplus [X_2]_k\oplus\cdots\oplus[X_{k+1}]_k.
$$
Here  $[X_i]_k$ is the span of the set of all $ X_j$, $ 2\le j \le n $ such that $ k $ divides $ i-j $. This easily follows if we choose $ k $ the least positive with $ d_2= d_{2+k}$.

The universal group of this grading is the factor group of the free abelian group generated by $ a, b $ by a single relation $ b= a^k b $, that is, the group $ \Z_k\times \Z $.

Any further coarsening of $\Gamma^l_0$ is clearly also the coarsening of a $\Gamma_k^0$, so we can restrict ourselves to only considering the coarsenings of the latter grading.

\textsc{Case 3.} Any proper coarsening of $\Gamma_k^0$, which does not decrease $ k $, is equivalent to one of the following.
$$
\Gamma_k^l: \g=[X_2]_k\oplus\cdots\oplus(\la X_1\ra\oplus[X_l]_k)\oplus\cdots\oplus[X_{k+1}]_k.
$$
Indeed, any further coarsening of $[X_l]_k$ will decrease $ k $ so we have to assume that $ d_1 = d_l $, for $2 \le l\le n $, proving our claim.  The universal group of this grading is the factor group of the group $\Z_k\times \Z $ by additional relation $ a= a^{l - 2}b $, that is, the group $\Z_k $.

Clearly, any further coarsening will lead to the decreasing of $ k $,  and so any grading is equivalent to one of the previous gradings.

Notice that these gradings are pairwise inequivalent. First, we have to look at the number of homogeneous components.  Then it becomes clear that we only need to distinguish between the gradings with different values of the superscript parameter $ l= 2, \ld,n $. In this case, if an automorphism $\vp$ maps $\Gamma^l_0$ to $\Gamma^m_0$, or $ \Gamma_k^l $ to $ \Gamma_k^m $, where $m>l$, then $\vp(\langle X_1,X_l\rangle)=\langle X_1,X_m\rangle$. But then $\vp(X_{l+1})=\vp([X_1,X_l])=[\vp(X_1),\vp(X_l)]=\alpha_{m+1}X_{m+1}$, $\vp(X_{l+2})=\alpha_{m+2}X_{m+2}$, etc. Finally, $\vp(X_{n-m+l+1})=0$, which is impossible because $2\le n-m+l+1\le n$.

As a result, we have the following.

\begin{theorem}\label{tL}
Let $\g$ be a filiform Lie algebra of the type $L_n$. If $\g$ is $G$-graded, then there exists a graded homogeneous adapted basis $\{X_1,X_2,\cdots,X_n\}$.  If $d_i$ denotes the degree of $X_i$, then  any $G$-grading is equivalent to one of the following pairwise inequivalent gradings:
\begin{enumerate}
\item $\Gamma_{\mathrm{st}}$, $U(\Gamma_{\mathrm{st}})=\Z^2$, $d_1=(1,0), \ d_i=(i-2,1)$, $i=2,\ld,n$.
\item $\Gamma_0^l$, $U(\Gamma_0^l)= \Z$, $2 \leq l \leq n,$ \ $d_1=1, \ d_i=i-l+1$, $i=2,\ld,n$.
\item $\Gamma_k^0$, $U(\Gamma_k^0)=\Z_k \times \Z$, $1 \leq k \leq n-2$,  $d_1=(\overline{1},0), \ d_i=(\overline{i-2},1)$, $i=2,\ld,n$.
\item $\Gamma_k^l$, $U(\Gamma_k^l)= \Z_k$, $1 \leq k \leq n-2, \ 2 \leq l \leq k+1$,  $d_1=\overline{1}, \ d_i=\overline{i-l+1}$, $i=2,\ld,n$.
\end{enumerate}
\end{theorem}

The total number of pairwise inequivalent gradings of $L_n$ is equal to $\displaystyle 1 + (n-1) + (n-2) + (1+2+\cdots+(n-2))=\frac{(n-1)(n+2)}{2}$.

In the case of $A_n^p$, we need to consider the coarsenings of the standard grading, which we denote here $\Gamma_{\mathrm{st}}$, where $U(\Gamma_{\mathrm{st}})\cong\Z$, with free generator $a$, $d_1=a$, $d_i=i+p-1$, where $i=2,\ld,n$. Clearly, in this case, any proper coarsening leads to relations $a^i=a^j$, hence $a^{i-j}=e$, for different $1\le i,j\le n$. If $m$ is the greatest common divisor of all such $i-j$ then the universal group is $\Z_m$. Any value of $m$ between $1$ and $n+p-2$ is possible. Indeed, if $p\le m\le n+p-2$ then $d_1=d_{m-p+2}$. If $1\le m\le n-2$ then $d_2=d_{m+2}$. But $1\le p\le n-4<n-2$, and so any $m$ between $1$ and $n+p-2$ is available. Let us denote the grading corresponding to $m$ by $\Gamma(m)$.
If $1\le m\le p-1$, $\Gamma(m)$ similar to $\Gamma_m^0$ of $L_n$:
$$
\Gamma(m): \g=\la X_1\ra\op[X_2]_m\oplus\cdots\oplus[X_{m+1}]_m.
$$
If $n-1\le m\le n+p-1$ and $l=m-p+2$, then  $\Gamma(m)$ similar to $\Gamma^l_0$ of $L_n$:
$$
\Gamma(m): \g=\la X_2\ra\oplus\cdots\oplus\la X_1,X_l\ra\oplus\cdots\oplus\la X_n\ra.
$$
If $p\le d\le n-2$ then we have the gradings similar to $\Gamma_m^l$ of $L_n$:
$$
\Gamma(m): \g=[X_2]_m\oplus\cdots\oplus(\la X_1\ra\oplus[X_l]_m)\oplus\cdots\oplus[X_{m+1}]_m.
$$
Here $l$ is a number between $2$ and $n$ such that $l+p-1\equiv 1\mod m$.
The pairwise inequivalence of all the above gradings follows, as in the case of $L_n$.

\begin{theorem}\label{tA}
Let $\g$ be a filiform Lie algebra of the type $A_n^p$. If $\g$ is $G$-graded, then there exists a graded homogeneous adapted basis $\{X_1,X_2,\cdots,X_n\}$.  If $d_i$ denotes the degree of $X_i$, then  any $G$-grading is equivalent to one of the following non equivalent gradings:
\begin{enumerate}
\item $\Gamma_{\mathrm{st}}$, $U(\Gamma_{\mathrm{st}})=\Z$, $d_1=1, \ d_i=p+i-1$, $i=2,\ld,n$.
\item $\Gamma(m)$, $U(\Gamma(m))= \Z_m$, $1 \leq m \leq n+p-2$, \ $d_1=\overline{1}, \ d_i=\overline{p+i-1}$, $i=2,\ld,n$.
\end{enumerate}
\end{theorem}

The total number of pairwise inequivalent gradings of $A_n^p$ is  equal to $n+p-2$.

\subsection{Gradings on $Q_n$ and $B_n^p$}\label{ssQB}

Let $\{Y_1,Y_2,\cdots,Y_n\}$ be a quasi-adapted basis of a filiform Lie algebra $\g$ of type $Q_n$ or $B_n^p$. Since $[Y_1,Y_i]=Y_{i+1}$ for $i=2,\cdots,n-2$, we know that $d_i=a^{i-2}b$ for $i=3,\cdots,n-1$. Also, $[Y_i,Y_{n-i+1}]=(-1)^{i+1}Y_n$ which assigns to $d_n$ the value of $d_n=a^{n-3}b^2$. If $\g$ is of the type $B_n^p$, we already have $b=a^{p+1}$. So in this case, $d_i=a^{p+i-1}$, for $2\le i\le n-1$, and $d_n=a^{n+2p-1}$. In the case of $Q_n$, the universal group of any grading is the factor group of the free abelian group with free basis $a,b$ by the relations satisfied by $a,b$ while in the case of $B_n^p$, this is a factor-group of the free abelian group of rank 1 or rank 2 but we have to consider, in the latter case, that $b=a^{p+1}$.

Let us first consider the case of $Q_n$.  We remember that $n=2m$, for some $m\ge 2$. Now any grading is a coarsening of the standard grading, which we denote by $\og_{\mathrm{st}}$. If all $d_ i$ are pairwise different then there is no coarsening and we have the standard grading.

\textsc{Case 1.} If $d_1=d_n $ then the grading is a coarsening of the grading
$$
\og(1,n): \g=\langle Y_2\rangle\oplus\cd\op\la Y_{n-1}\ra\op\la Y_1,Y_n\ra.
$$
Since this is indeed a grading of $\g$, our claim follows. The universal group of this grading is the factor group of the free abelian group generated by $a, b$ by a single relation $a= a^{n - 3}b^2$, that is, the group $\Z\times \Z_2$.

\textsc{Case 2.} If $d_1=d_l $, for $2\le l\le n-1$ then $a=a^{l-2}b$. In this case also $d_n=a^{n-3}b^2=a^{n-l}b=d_{n-l+2}$. Hence, for all $q$ satisfying $l+q=n+2$, except $l=2$, or $q=2$ the grading is a coarsening of one of the following gradings. If $l\neq q$ then we have
$$
\og^l_0: \g=\langle Y_2\rangle\oplus\cd\op\la Y_1,Y_l\ra\op\cd\op\la Y_m,Y_n\ra\op\la Y_{n-1}\ra.
$$
If $l = q = m+1$ then we have
$$
\og^{m+1}_0: \g=\langle Y_2\rangle\oplus\cd\op\la Y_1,Y_{m+1},Y_n\ra\op\cd\op\la Y_{n-1}\ra.
$$
Notice that this grading is a coarsening of $\og(1,n)$.

In the exceptional cases, $l=2$ or $q=2$, we have the following.

If $l=2$ then we have
$$
\og^2_0: \g=\langle Y_1,Y_2\rangle\oplus\cd\op\la Y_1,Y_l\ra\op\cd\op\la Y_n\ra.
$$

If $q=2$ then we have
$$
\og^n_0: \g=\langle Y_2,Y_n\rangle\op\la Y_3\ra\op\cd\op\la Y_{n-1}\ra\op\la Y_1\ra.
$$
Since all these is indeed gradings of $\g$, our claim follows. The universal group of this grading is the factor group of the free abelian group generated by $a, b$ by a single relation $a= a^{l - 2}b $, that is, the group $\Z$.

\textsc{Case 3}. If $d_i=d_j$, for $2\le i<j\le n-1 $ then the grading is a coarsening of a grading
$$
\og_k^0:\g=\la Y_1\ra\oplus [Y_2]_k\oplus\cdots\oplus[Y_{k+2}]_k\op\la Y_n\ra,
$$
where $k$ is such that $1\le k\le n-3$.
Here  $[Y_i]_k$ is the span of the set of all $ Y_j$, $ 2\le j \le n $ such that $ k $ divides $ i-j $. This easily follows if we choose $ k $ the least positive with $ d_2= d_{2+k}$.

The universal group of this grading is the factor group of the free abelian group generated by $a, b$ by a single relation $b=a^k b$, that is, the group $ \Z_k\times \Z $.

Any coarsening of $\og(1,n)$ (Case 1) is either $\og^{m+1}_0$ or is a coarsening of some $\og_k^0$. Any coarsening of a grading $\og^l_0$ (Case 2) is also a coarsening of some $\og_k^0$, so we can restrict ourselves to only considering the coarsenings of the latter gradings. This shows that any grading is either standard, or equivalent to one of the gradings in Cases 1 - 3 or is a proper coarsening of a grading $\og_k^0$. Let us choose $\og_k^0$ with minimal possible $k$.

\textsc{Case 4.} Any proper coarsening of $\og_k^0$, which does not change $ k $, is equivalent to one of the following.
$$
\og(1,n)_k:\g= [Y_2]_k\oplus\cdots\oplus[Y_{k+2}]_k\op\la Y_1, Y_n\ra,
$$
where $1\le k\le n-3$, or
$$
\og_k^l: \g=[Y_2]_k\oplus\cdots\oplus(\la Y_1\ra\oplus[Y_l]_k)\oplus\cdots\oplus([Y_q]_k\oplus\la Y_n\ra)\op\cd\op[Y_{k+1}]_k,
$$
where $1\le k\le n-3$, $2\le l,q\le k+1$, $l+q\equiv n+2\mod k$.

The universal group of $\og(1,n)_k$ is $\Z_k\op\Z_2$, whereas, in the case of $\og_k^l$, the universal group is $\Z_k$.

Indeed, any further coarsening of $[Y_l]_k$ will decrease $ k $ so we have to assume that $ d_1 = d_l $, for $2 \le l\le n $, proving our claim.

Clearly, any further coarsening will lead to further decreasing of $ k $,  and so any grading is equivalent to one of the previous gradings.

Notice that these gradings are pairwise inequivalent. First, the gradings with different universal groups are not equivalent. As a result, we only need to distinguish between the gradings in the sets $\og^l_0$, $2\le l\le n$ (the universal group $\Z$) and $\og^l_k$, $2\le l\le k+1$ (the universal group $\Z_k$). This is done exactly in the same way, except the cases where $Y_1$ is in a component which is not 2-dimensional. However, such grading can only be mapped to a grading with the same property because $Y_1$ is the only element among $Y_i$ with $(\ad Y_i)^{n-2}\neq 0$.

As a result, we have the following.

\begin{theorem}\label{tQ}
Let $\g$ be a filiform Lie algebra of the type $Q_n$. If $\g$ is $G$-graded, then there exists a graded homogeneous quasi-adapted basis $\{Y_1,Y_2,\ldots,Y_n\}$.  If $d_i$ denotes the degree of $Y_i$, then  any $G$-grading is equivalent to one of the following pairwise equivalent gradings:
\begin{enumerate}
\item $\og_\mathrm{st}$, $U(\og_\mathrm{st})=\Z^2$, $d_1=(1,0), \ d_i=(i-2,1)$, $2\le i\le n-1$, $d_n=(n-3,2)$.
\item $\og(1,n)$, $U(\og(1,n))= \Z\times\Z_2$, $d_1=(1,\overline{0})=d_n$, $d_i=(i-2,\overline{1})$, $2\le i\le n-1$.
\item $\og_0^l$, $U(\og_0^l)= \Z$, $2 \leq l \leq n$,  $d_1=1$, $d_i=i-l+1$, $2\le i\le n-1$, $d_n=n-2l+3$.
\item $\og_k^0$, $U(\og_k^0)=\Z_k \times \Z$, $1 \leq k \leq n-3$, $d_1=(\overline{1},0), \ d_i=(\overline{i-2},1),d_n=(\overline{n-3},2)$.
\item $\og(1,n)_k$, $U(\og(1,n)_k)=\Z_k \times \Z_2$, $1 \leq k \leq n-3$, $d_1=(\overline{1},\overline{0})=d_n$, \ $d_i=(\overline{i-2},\overline{1})$, $2\le i\le n-1$.
\item $\og_k^l$, $U(\og_k^l)= \Z_k$, $1 \leq k \leq n-3, \ 2 \leq l \leq k+1$,  $d_1=\overline{1}$, \ $d_i=\overline{i-l+1}$, $d_n=\overline{n-2l+3}$.
\end{enumerate}
\end{theorem}

The total number of pairwise inequivalent gradings of $Q_n$ is equal to $\displaystyle 1 + 1+(n-1) + (n-3) + (n-3)+(1+2+\cdots+(n-3))=\frac{(n-1)(n+2)}{2}-1$.

In the case of $B_n^p$, we need to consider the coarsenings of the standard grading, which we denote here $\Gamma_{\mathrm{st}}$, where $U(\Gamma_{\mathrm{st}})\cong\Z$, with free generator $a$, $d_1=a$, $d_i=i+p-1$, where $i=2,\ld,n-1$, $d_n=n+2p-1$. Clearly, in this case, any proper coarsening leads to relations $a^i=a^j$, hence $a^{i-j}=e$, for different $1\le i,j\le n$. If $m$ is the greatest common divisor of all such $i-j$ then the universal group is $\Z_m$. Any value of $m$ between $1$ and $n+p-3$ is possible, similar to the case of $A_n^p$. One more possible isolated value for $m$ appears if we choose $d_1=d_n$. Then $m=n+2p-3$.  Let us denote the grading corresponding to $m$ by $\og(m)$.

If $1\le m\le p-1$, $\og(m)$ similar to $\og_m$ of $Q_n$:
$$
\og(m): \g=\la Y_1\ra\op[Y_2]_m\oplus\cdots\oplus[Y_{m+1}]_m\op\la Y_n\ra.
$$

If $p\le m\le n-3$ then we have the gradings similar to $\og_m^l$ of $Q_n$:
$$
\og(m): \g=[Y_2]_m\oplus\cdots\oplus(\la Y_1\ra\oplus[Y_l]_m)\oplus\cdots\oplus([Y_q]_m\oplus\la Y_n\ra)\op\cd\op[Y_{k+1}]_m.
$$

If $n-2\le m\le n+p-3$, $\og(m)$ similar to $\og^l$ of $Q_n$:
$$
\og(m): \g=\langle Y_2\rangle\oplus\cd\op\la Y_1,Y_l\ra\op\cd\op\la Y_m,Y_n\ra\op\la Y_{n-1}\ra.
$$
Here $l=m-p+2$.

If $m=n+2p-3$ then $\og(m)$ is similar to $og(1,n)$:
$$
\og(n+2p-3): \g=\langle Y_2\rangle\oplus\cd\op\la Y_{n-1}\ra\op\la Y_1,Y_n\ra.
$$

The pairwise inequivalence of all the above gradings follows, as in the case of $Q_n$.

\begin{theorem}\label{tB}
Let $\g$ be a filiform Lie algebra of the type $B_n^p$. If $\g$ is $G$-graded, then there exists a graded homogeneous adapted basis $\{Y_1,Y_2,\cdots,Y_n\}$.  If $d_i$ denotes the degree of $Y_i$, then  any $G$-grading is equivalent to one of the following non equivalent gradings:
\begin{enumerate}
\item $\og_{\mathrm{st}}$, $U(\og_{\mathrm{st}})=\Z$, $d_1=1, \ d_i=p+i-1$, $i=2,\ld,n-1$, $d_n=n+2p-2$.
\item $\og(m)$, $U(\og(m))= \Z_m$, $1 \leq m \leq n+p-3$ or $m=n+2p-3$, \ $d_1=\overline{1}, \ d_i=\overline{p+i-1}$, $i=2,\ld,n$, $d_n=\overline{n+2p-2}$.
\end{enumerate}
\end{theorem}

The total number of pairwise inequivalent gradings of $B_n^p$ is equal to $n+p-3$.

\subsection{Characteristically nilpotent Lie algebras}
It is known from \cite{Goze-Khakim} that any filiform $(n+1)$-dimensional  Lie algebra over an algebraic field of characteristic $0$ is defined by its Lie bracket $\mu$ with $\mu=\mu_0 + \psi$ where $\mu_0$ is the Lie multiplication of $L_{n+1}$ and $\psi$ a $2$-cocycle of $Z^2(L_{n+1},L_{n+1})$ satisfying   $\psi\circ \psi=0$, that is, $\psi$ is also a $(n+1)$-dimensional Lie multiplication. Let us consider the natural $\Z$-grading of $L_{n+1}$:
$$L_{n+1}= \bigoplus_{i \in \Z} L_{n+1,i}$$
where $L_{n+1,1}$ is generated by ${e_0,e_1}$ and $L_{n+1,i}$ by ${e_i}$ for $i=2, \ldots, n,$ and other subspaces are zero. This grading induces a $\Z$-grading in the spaces of cochains of the Chevalley-Eilenberg complex of $L_{n+1}$:
$$\mathcal{C}_p^ k(L_{n+1},L_{n+1})=\{\phi \in \mathcal{C}_k(L_{n+1},L_{n+1}), \ \phi(L_{n+1,i_1},\ldots,L_{n+1,i_k})\subset L_{n+1,i_1+\ldots+i_k+p}\}.$$
Since $d(\mathcal{C}_p^ k(L_{n+1},L_{n+1})) \subset \mathcal{C}_p^ {k+1}(L_{n+1},L_{n+1})$, we deduce a grading in the spaces of cocycles and coboundaries. Let
$$H_p^ k(L_{n+1},L_{n+1})=Z_p^ k(L_{n+1},L_{n+1})/B_p^ k(L_{n+1},L_{n+1})$$
the corresponding grading in the Chevalley-Eilenberg cohomological spaces of $L_{n+1}$. We put
$$F_0H^ k(L_{n+1},L_{n+1})=\bigoplus_{p \in Z}H_p^ k(L_{n+1},L_{n+1}), \ \ F_1H^ k(L_{n+1},L_{n+1})=\bigoplus_{p \geq 1}H_p^ k(L_{n+1},L_{n+1}).$$
We have:
\begin{proposition}
Let $\psi_{k,s}$, $1 \leq k \leq n-1, \ 2k \leq s \leq n$  be the $2$-cocycle in $\mathcal{C}^ 2(L_{n+1},L_{n+1})$ defined by
\begin{itemize}
\item $\psi_{k,s}(e_k,e_{k+1})=e_s$,
\item $\psi_{k,s}(e_i,e_{i+1})=0$ if $i \neq k$,
\item $\psi_{k,s}(e_i,e_{j})=0$ if $i >  k$,
\item $\psi_{k,s}(e_i,e_{j})=(-1)^{k-i}C_{j-k-1}^{k-i}e_{i+j+s-2k-1}, \ 1 \leq i  \leq k < j-1 \leq n-1, \ 0 \leq i+j-2k-1 \leq n -s$
\end{itemize}
Then the family of $\psi_{k,s}$ with $1 \leq [n/2]-1, \ 4 \leq s \leq n$ forms a basis of $F_1H^ 2(L_{n+1},L_{n+1})$.
\end{proposition}
Note that the cocycles $\psi_{k,s}$  also satisfy $\psi_{k,s}(e_k,e_j)=e_{j+s-k-1} $ when $k<j$.
\begin{proposition}
If $\mu=\mu_0+\psi$ is the Lie multiplication of a filiform $(n+1)$-dimensional Lie algebra, then $[\psi]  \in F_1H^ 2(L_{n+1},L_{n+1})$ if $n$ is even or $\psi \in F_1H^ 2(L_{n+1},L_{n+1})+[\psi_{(n-1)/2,n}]$ if $n$ is odd where $[\psi]$ denote the class in $H^ 2(L_{n+1},L_{n+1})$ of the $2$-cocycle $\psi$.
\end{proposition}
\noindent{\bf Examples}
\begin{enumerate}
\item Any $7$-dimensional filiform Lie algebra can be written as $\mu=\mu_0+\psi$ with
$$\psi=a_{1,4}\psi_{1,4}+a_{1,5}\psi_{1,5}+a_{1,6}\psi_{1,6}+a_{2,6}\psi_{2,6}.$$
\item Any $8$-dimensional filiform Lie algebra can be written as $\mu=\mu_0+\psi$ with
$$\psi=a_{1,4}\psi_{1,4}+a_{1,5}\psi_{1,5}+a_{1,6}\psi_{1,6}+a_{1,7}\psi_{1,7}+a_{2,6}\psi_{2,6}+a_{2,7}\psi_{2,7}+a_{3,7}\psi_{3,7}.$$
\item Any $9$-dimensional filiform Lie algebra can be written as $\mu=\mu_0+\psi$ with
$$
\begin{array}{lll}
\psi&=&a_{1,4}\psi_{1,4}+a_{1,5}\psi_{1,5}+a_{1,6}\psi_{1,6}+a_{1,7}\psi_{1,7}+a_{1,8}\psi_{1,8}+a_{2,6}\psi_{2,6}+a_{2,7}\psi_{2,7}+a_{2,8}\psi_{2,8}\\
&&+a_{3,8}\psi_{3,8}.
\end{array}
$$
\item Any $10$-dimensional filiform Lie algebra can be written as $\mu=\mu_0+\psi$ with
$$
\begin{array}{lll}
\psi&=&a_{1,4}\psi_{1,4}+a_{1,5}\psi_{1,5}+a_{1,6}\psi_{1,6}+a_{1,7}\psi_{1,7}+a_{1,8}\psi_{1,8}+a_{1,9}\psi_{1,9}+a_{2,6}\psi_{2,6}+a_{2,7}\psi_{2,7}\\
&&+a_{2,8}\psi_{2,8}+a_{2,9}\psi_{2,9}+a_{3,8}\psi_{3,8}+a_{3,9}\psi_{3,9}+a_{4,9}\psi_{4,9}.
\end{array}
$$
\end{enumerate}
To recognize characteristically nilpotent Lie algebras among the Lie filiform Lie algebra, we can use the notion of a \textit{sill} algebra. Recall that a filiform Lie algebra $\g$
such that $\gr(\g)$ is isomorphic to $L_{n+1}$ can be written in an adapted basis as
$$[e_0,e_i]=e_{i+1}, \ i=1\cdots,n-1, \ \ [e_i,e_j]=\sum\limits_{r=1}^{n-i-j}a_{ij}^re_{i+j+r}, \ 1 \leq i < j \leq n-2.$$
A filiform Lie algebra $\g$
such that $\gr(\g)$ is isomorphic to $Q_{n+1}$ is written in an quasi-adapted basis:
$$[Z_0,Z_i]=Z_{i+1}, \ i=1\cdots,n-2, \ \ [Z_i,Z_{n-i}]=(-1)^iZ_n, \ \ [Z_i,Z_j]=\sum\limits_{r=1}^{n-i-j}b_{ij}^rZ_{i+j+r}$$ for $\ 1 \leq i < j \leq n-2.$
\begin{definition}\cite{G.K}
Let $\g$ be a $(n+1)$-dimensional filiform Lie algebra such that $\gr(\g)$ is isomorphic to $L_{n+1}$. The sill algebra of $\g$ is defined by
$$[e_0,e_i]=e_{i+1}, \ i=1,\cdots,n-1, \ \ [e_i,e_j]=a_{ij}^re_{i+j+r}$$
where $r\neq 0$ is the smallest index such that $a_{ij}^r \neq 0$ for some $(i,j)$.

If $\gr(\g)$ is isomorphic to $Q_{n+1}$, then the sill algebra is defined by

\noindent (1) If $b_{ij}^{n-i-j}=0$,
$$[Z_0,Z_i]=Z_{i+1}, \ i=1,\cdots,n-2, \ \ [Z_i,Z_{n-i}]=(-1)^iZ_n, \ \ [Z_i,Z_j]=b_{ij}^rZ_{i+j+r}$$ for $\ 1 \leq i < j \leq n-2$
where $r\neq 0$ is the smallest index such that $b_{ij}^r \neq 0$ for some $(i,j)$.

\noindent (2) If $b_{ij}^{n-i-j}\neq 0$ for some $(i,j)$,
$$[Z_0,Z_i]=Z_{i+1}, \ i=1,\cdots,n-2, \ \ [Z_i,Z_{n-i}]=(-1)^iZ_n, \ \ [Z_i,Z_j]=b_{ij}^{n-i-j}Z_{n}$$ for $\ 1 \leq i < j \leq n-2.$
\end{definition}
\begin{proposition}
A filiform Lie algebra is characteristically nilpotent if and only if it is not isomorphic to its sill algebra.
\end{proposition}

\noindent{\bf Examples}
\begin{enumerate}
\item Any $7$-dimensional filiform characteristically nilpotent Lie algebra can be written $\mu=\mu_0+\psi$ with
\begin{enumerate}
\item $\psi=\psi_{1,5}+\psi_{1,6},$
\item $\psi=\psi_{1,4}+\psi_{1,6},$
\item $\psi=\psi_{1,5}+\psi_{2,6}.$
\end{enumerate}
\item Any $8$-dimensional filiform characteristically nilpotent Lie algebra Lie algebra can be written as $\mu=\mu_0+\psi$ with
\begin{enumerate}
\item $\psi=\psi_{1,4}+a_{1,5}\psi_{1,5}-a_{2,6}\psi_{2,6}+\psi_{3,7}.$
\item $\psi=\psi_{1,5}+a_{1,6}\psi_{1,6}+\psi_{3,7}.$
\item $\psi=a_{1,4}\psi_{1,4}+\psi_{2,6}+\psi_{2,7}.$
\item $\psi=\psi_{1,5}+\psi_{2,6}.$
\item $\psi=\psi_{1,4}+a_{1,6}\psi_{1,6}+\psi_{2,7}.$
\item $\psi=a_{1,5}\psi_{1,5}+\psi_{1,6}+\psi_{2,7}.$
\item $\psi=a_{1,4}\psi_{1,4}+\psi_{1,6}+\psi_{1,7}.$
\item $\psi=\psi_{1,4}+\psi_{1,6}.$
\item $\psi=\psi_{1,4}+\psi_{1,7}.$
\item $\psi=\psi_{1,5}+\psi_{1,6}.$
\end{enumerate}
\end{enumerate}

\subsection{$\Z_2$-gradings of filiform characteristically nilpotent Lie algebras}\label{ssz2}

Since characteristically nilpotent Lie algebras have zero rank, they do not admit $\Z$-gradings. The following shows that there exists a class of such Lie algebras admitting $\Z_2$-gradings.
\begin{proposition}\label{charanilp}
Let $\mu_0+\psi$ the multiplication of a $(n+1)$-dimensional characteristically nilpotent Lie algebra $\g$ such that $\gr(\g)$ is isomorphic to $L_{n+1}$. If
$$\psi=\sum a_{k,2s}\psi_{k,2s}$$
or $$\psi=\sum a_{k,2s+1}\psi_{k,2s+1}$$
then this Lie algebra admits a $\Z_2$-grading.
\end{proposition}
\begin{proof} Since $\psi_{k,2s}(e_i,e_j)=ae_{i+j+2s-2k-1}$, the Lie algebra $\mu_0+\psi$ is characteristically nilpotent as soon as we have in the sum $\psi$ two terms $\psi_{k,2s}$ and $\psi_{k',2s'}$ such that $s-k \neq s'-k'$. Let us consider two $\Z_2$-gradings of the vector space $\g$ as follows.
\begin{itemize}
\item $\g=\la e_1,e_3,\cdots,e_{2p \pm 1}\ra\bigoplus \la e_0,e_2,e_4,\cdots,e_{2p}\ra $
\item $\g=\la e_2,e_4,\cdots,e_{2p}\ra\bigoplus \la e_0,e_1,e_3,\cdots,e_{2p \pm 1}\ra$
\end{itemize}
(the $\pm$ sign means that we consider the cases $n$ odd and $n $ even at one and the same time).

We consider the first vectorial decomposition. From our description of gradings on $L_n$ we can see that this is a grading of $\mu_0$. It is sufficient then to check that this is a grading of $\psi_{k,2s}$.
A cocycle $\psi_{k,s}$ is homogeneous, that is satisfies $\psi_{k,s}(\g_i,\g_j) \subset \g_{i+j(\!\!\mod\, 2)}$ if $s$ is even. In fact
$$\psi_{k,s}(e_{2i+1,2j+1})=\lambda e_{2i+2j-2k+1+s}$$
where $\lambda$ is a non zero constant and $2i+2j-2k+1+s$ is odd if and only if $s$ is even. Likewise
$$\psi_{k,s}(e_{2i,2j})=\lambda e_{2i+2j-2k-1+s}$$
with $\lambda \neq 0$ and $2i+2j-2k-1+s$ is odd if and only if $s$ is even. Since also
$$\psi_{k,s}(e_{2i},e_{2j+1})=e_{2i+2j-2k+s},$$
$2i+2j-2k+s$ is even as soon as $s$ is even. Thus the existence of this grading implies that $s$ is even.

Similarly, the vectorial decomposition of the second type, is a $\Z_2$-grading of $\mu_0+\psi_{k,s}$ if $s$ is odd.
\end{proof}

\begin{proposition}
Let $\mu_0+\psi$ the multiplication of a $(n+1)$-dimensional characteristically nilpotent Lie algebra $\g$ such that $\gr(\g)$ is isomorphic to $Q_{n+1}$. If
$$\psi=\sum a_{k,2s}\psi_{k,2s}+\psi_{\frac{n-1}{2},n}$$
or $$\psi=\sum a_{k,2s+1}\psi_{k,2s+1}+\psi_{\frac{n-1}{2},n}$$
then this Lie algebra admits a $\Z_2$-grading.
\end{proposition}
\pf Since any $\Z_2$-grading on  $\g$  induces the same grading on $\gr(\g)=Q_{n+1}$, we consider the vectorial decompositions of $\g$:
\begin{itemize}
  \item $\g=\la e_2,e_4,\cdots,e_{n-1}\ra\oplus \la e_0,e_1,e_3,\cdots,e_{n}\ra,$
  \item $\g=\la e_0+e_1,e_{n}\ra\oplus \la e_1,e_2,\cdots,e_{n-1}\ra,$
  \item $\g=\la e_1,e_3,\cdots,e_{n-2}\ra\oplus \la e_0+e_1,e_2,e_4,\cdots,e_{n-1},e_n\ra$
\end{itemize}
If the multiplication of $\g$ is given by $\mu_0+ \psi$ with $\psi=\sum a_{k,2s+1}\psi_{k,2s+1}+\psi_{\frac{n-1}{2},n}$, then the first vectorial decomposition is also a $\Z_2$-grading. If the multiplication of $\g$ is given by $\mu_0 + \Psi$ with $\psi=\sum a_{k,2s}\psi_{k,2s}+\psi_{\frac{n-1}{2},n}$, then the third vectorial decomposition is also a $\Z_2$-grading.  If the second vectorial decomposition is a grading, then  $\g$ is not characteristically nilpotent.

\begin{corollary}
There exists an infinite family of graded characteristically nilpotent filiform Lie algebras.
\end{corollary}
\pf  We consider the $9$-dimensional filiform Lie algebras given by
$$\mu=\mu_0+\psi_{1,4}+\alpha \psi_{2,6}+\psi_{2,8}+\frac{3\alpha^2}{\alpha+2}\psi_{3,8}$$
with $\alpha \neq 0$ and $-2$. These Lie algebras admits a $\Z_2$-grading and for two different values of $\alpha$ we have non isomorphic Lie algebras \cite{gomez}. Moreover, since $\alpha \neq 0$, these Lie algebras are characteristically nilpotent.

\subsection{$\Z_k$-gradings, $ k >2$, of filiform characteristically nilpotent Lie algebras}.

We assume that $k >2$. A $\Z_k$-grading of $L_n$ is equivalent to one of the following
$$\Gamma_k^l: L_n=\sum_{i=2}^{l-1}[X_i]_k \oplus (\la X_1\ra \oplus [X_l]_k) \oplus  \sum_{j=l+1}^{k+1}[X_j]_k$$
where $l $ is a parameter satisfying $2 \leq l \leq k+1$. The homogeneous component of this grading corresponding to the identity of $\Z_k$ is $[X_{l-1}]_k$. Two cocycles $\psi_{h_1,s_1}$ and $\psi_{h_2,s_2}$ send an homogeneous component (in particular $[X_{l-1}]_k$) in another homogeneous component if and only if
$$s_1-2h_1=s_2-2h_2 (\mod \ k).$$
We deduce
\begin{proposition}
Any filiform characteristically nilpotent Lie algebra $\g$ such that $\gr(\g)$ is isomorphic to $L_{n+1}$ whose Lie multiplication is of the form
 $$\mu=\mu_0+\sum_{i\in I} a_{h_i,s_i}\psi_{h_i,s_i}$$
 with $$s_i-2h_i=s_j-2h_j (\mod \ k), \ k<n-2$$
 for any $i,j \in I$
admits a $\Z_k$-grading.
\end{proposition}{
\begin{proposition}
For any $k$, there exists a $\Z_k$-graded characteristically nilpotent filiform Lie algebra.
\end{proposition}
\pf If fact, we consider the  filiform Lie algebra of dimension $n=k+5$ given by
$$\mu=\mu_0+\psi_{1,4}+\psi_{1,4+k}.$$
that is $$
\left\{
\begin{array}{l}
\mu(e_0,e_i)=e_{i+1}, \ i=2,\cdots,k+3,\\
\mu(e_1,e_2)=e_{4}+e_{k+4},\\
\mu(e_1,e_i)=e_{i+2}, \ i=3,\cdots,k+2.
\end{array}
\right.
$$
The Jacobi conditions are satisfied. The sill algebra is given by $\mu_0+\psi_{1,4}$ and is not isomorphic to $\mu$ as soon as $k \neq 0$. Then this Lie algebra is characteristically nilpotent and from the previous proposition, it is $\Z_k$-graded.

\end{document}